\documentclass[final,3p,times,12pt]{elsarticle}

\usepackage{amssymb}
\usepackage{multicol}
\usepackage[svgnames]{xcolor}

\journal{Computer Methods in Applied Mechanics and Engineering (CMAME)}

\usepackage{amsmath}
\usepackage{amsfonts}
\usepackage{amsthm}
\usepackage{mathtools}
\usepackage{subfig}
\usepackage{lineno}
\usepackage[hidelinks]{hyperref}
\usepackage{textcomp}
\usepackage{algorithm,algpseudocode,algorithmicx}

\usepackage{float}
\usepackage{array,multirow}
\usepackage{color}
\usepackage{tikz}
\usetikzlibrary{patterns}
\usetikzlibrary{spy}
\usetikzlibrary{spy,decorations.fractals,shapes.misc}

\usepackage{pgfplots}
\pgfplotsset{compat=1.16}
\newsavebox\Axis

\definecolor{lightgray}{gray}{0.80}

\usetikzlibrary{decorations.pathreplacing}
\usepackage[listings,theorems,skins,breakable]{tcolorbox}
\tcbset{skin=enhanced}
\newtcolorbox{lbracebox}[1][Word]{%
   frame hidden,enlarge left by=2cm,width=\linewidth-2cm,%
  overlay unbroken = {\draw [decorate,decoration={brace,amplitude=10pt},]%
                     (frame.south west)-- (frame.north west)
                    node [black,midway,left,xshift=-.6cm] {#1};},% 
}

\newcommand{\Bezier}{B\'ezier~}

\def\B{}
\def\Bl{}

\definecolor{grey1}{rgb}{0.5, 0.5, 0.5}
\definecolor{green1}{rgb}{0.4660, 0.6740, 0.1880} 
\definecolor{blue1}{rgb}{0, 0.4470, 0.7410} 
\definecolor{red1}{rgb}{0.8500, 0.3250, 0.0980}
\definecolor{yellow1}{rgb}{0.9290, 0.6940, 0.1250}
\definecolor{purple1}{rgb}{0.4940, 0.1840, 0.5560}
\definecolor{lightblue1}{rgb}{0.3010, 0.7450, 0.9330}
\definecolor{bordeaux1}{rgb}{0.6350, 0.0780, 0.1840}

\definecolor{burntorange}{rgb}{0.74902,0.341176,0}

\theoremstyle{plain}
\newtheorem{theorem}{Theorem}[section]
\newtheorem{proposition}[theorem]{Proposition}
\newtheorem{lemma}[theorem]{Lemma}
\newtheorem{corollary}[theorem]{Corollary}

\theoremstyle{definition}
\newtheorem{definition}[theorem]{Definition}

\newcommand{\vect}[1]{\boldsymbol{#1}} 								% a vector
\newcommand{\mat}[1]{\boldsymbol{#1}} 								% a vector
\newcommand{\mmat}[1]{\mathbf{#1}} 									% a vector
\newcommand{\List}[1]{\left( \right. #1 \left.\right) }				% list of items
\newcommand{\bilinearform}[3]{#1(#2, #3)} 							% bilinear form
\newcommand{\pairing}[2]{\langle#1, #2 \rangle} 					% duality pairing
\newcommand{\grammian}{G}													% Grammian matrix notation

\newcommand{\SPD}[1]{\mathbb{SPD}^{#1}} 							% space of SPD matrices

\newcommand{\nsd}{d}																% number of space dimensions
\newcommand{\domain}{\Delta}													% parameter domain 1D
\newcommand{\Domain}{\Omega}												% real domain domain multi-d
\newcommand{\ParamDomain}{\hat{\Domain}}									% parameter domain multi-d
\newcommand{\nelms}{\delta}													% number of elements
\newcommand{\ndofs}{n}															% number of elements
\newcommand{\N}{N}																	% dimension
\newcommand{\x}{x}																	% parameter
\newcommand{\p}{p}																	% polynomial degree
\renewcommand{\r}{r}																% regularity
\newcommand{\bp}[1]{\ensuremath{\param_{#1}}}
\newcommand{\knot}[1]{\ensuremath{t_{#1}}}
\newcommand{\param}{\ensuremath{x}}
\newcommand{\knots}{\ensuremath{\Xi}} 									% Knotvector
\newcommand{\bspline}[1]{B_{#1}}											% univariate B-spline
\newcommand{\dual}[1]{\tilde{B}_{#1}}									% dual functionals
\newcommand{\approxdual}[1]{\hat{B}_{#1}}							% approximate dual functionals

\newcommand{\idx}[1]{\textrm{#1}}											% roman literals
\newcommand{\trialfun}[1]{N_{\textrm{#1}}}							% multidimensional B-spline
\newcommand{\testfun}[1]{N_{\textrm{#1}}}	    					% multidimensional dual functionals
\newcommand{\dualtestfun}[1]{\hat{N}_{\textrm{#1}}}	    	% multidimensional dual functionals

\newcommand{\funspace}{\mathcal{V}}										% function space

\newcommand{\sspace}[2]{\mathbb{S}^{#1}_{#2}}				% univariate B-spline
\newcommand{\pspace}[1]{\mathbb{P}^{#1}}							% space of polynomials
\newcommand{\Span}[1]{\text{span} \left( #1 \right)}				% span of spline space
\modulolinenumbers[5]

\begin{document}

\begin{frontmatter}

\title{Higher order accurate mass lumping for explicit isogeometric methods based on approximate dual basis functions}

\corref{cor1}
\author[address1]{Ren\'e R. Hiemstra}
\ead{hiemstra@mechanik.tu-darmstadt.de}

\author[address1]{Thi-Hoa Nguyen}
\ead{nguyen@mechanik.tu-darmstadt.de}

\author[address3]{Sascha Eisentr\"ager}
\ead{sascha.eisentraeger@ovgu.de}

\author[address2]{Wolfgang Dornisch}
\ead{wolfgang.dornisch@b-tu.de}

\author[address1]{Dominik Schillinger}
\ead{schillinger@mechanik.tu-darmstadt.de}

\cortext[cor1]{Corresponding author}
\address[address1]{Institute for Mechanics, Computational Mechanics Group, Technical University of Darmstadt, Germany}

\address[address2]{Chair of Structural Analysis and Dynamics, Brandenburg University of Technology Cottbus-Senftenberg, Germany}

\address[address3]{\Bl Institute of Mechanics, Chair of Computational Mechanics, Otto von Guericke University Magdeburg, Germany \B}

\begin{abstract}
This paper introduces a mathematical framework for explicit structural dynamics, employing approximate dual functionals and rowsum mass lumping. We demonstrate that the approach may be interpreted as a Petrov-Galerkin method that utilizes rowsum mass lumping or as a Galerkin method with a customized higher-order accurate mass matrix. Unlike prior work, our method correctly incorporates Dirichlet boundary conditions while preserving higher order accuracy. The mathematical analysis is substantiated by spectral analysis and a two-dimensional linear benchmark that involves a non-linear geometric mapping. Our results reveal that our approach achieves accuracy and robustness comparable to a traditional Galerkin method employing the consistent mass formulation, while retaining the explicit nature of the lumped mass formulation.
\end{abstract}

%\begin{highlights} \Bl 
%	\item Higher order accurate mass lumping for explicit isogeometric methods.
%	\item Consistent application of Dirichlet boundary conditions
%	\item Application to linear explicit analysis employing outlier removal.
%	\item Favourable spectral properties and good critical time-step values. \B
%\end{highlights}


\end{frontmatter}

%\tableofcontents

%\linenumbers

\section{Introduction}
Explicit methods for structural dynamics advance or update the solution at each time increment according to an evolution equation represented by the matrix problem
\begin{align}
	\mmat{M}\, \ddot{\vect{u}}_{n+1} = \vect{F}(\vect{u}_n, \dot{\vect{u}}_n, \ddot{\vect{u}}_n) \; .
	\label{eq:explicit}
\end{align}
Here, $\mmat{M} \in \mathbb{R}^{\N \times \N}$ is the mass matrix (which is a sparse, symmetric and positive definite matrix), $\ddot{\vect{u}}_{n+1} \in \mathbb{R}^{\N}$ is a vector of unknown accelerations at time $t_{n+1}$, and $\vect{F} \in \mathbb{R}^\N$ encapsulates the interior and exterior forces, which are a function of the known displacement, $\vect{u}_n$, velocity, $\dot{\vect{u}}_n$, and acceleration,  $\ddot{\vect{u}}_n$, at time $t_{n}$.

State-of-the-art commercial solvers such as LS-Dyna\footnote{\url{https://www.ansys.com/products/structures/ansys-ls-dyna}}, RADIOSS\footnote{\url{https://altair.com/radioss}}, or PAM-CRASH\footnote{\url{https://www.esi-group.com/pam-crash}} are all based on the linear finite element method \cite{belytschko1984explicit,benson1992computational,hughes2012finite} in conjunction with explicit second-order accurate time integration methods such as the central difference method or variations thereof \cite{hughes2012finite}. In typical applications such as the analysis of vehicle crashworthiness and the simulation of metal forming or stamping, the internal force vector of the discrete system \eqref{eq:explicit} involves shell elements, nonlinear material models, and contact with friction, for which the linear finite element provides a well-developed technological setting. A key component is the replacement of the consistent mass matrix $\mmat{M}$ with a diagonal matrix, $\mmat{M}_L$, obtained by ``mass lumping''. It results in a simple and highly computationally and memory efficient solution update, while maintaining optimal spatial accuracy of the linear finite element method. In addition, the lumped mass matrix reduces the highest frequencies of the discretization, thereby increasing the critical time-step size. These codes therefore achieve an extreme level of computational and memory efficiency.

Initiated in 2005, isogeometric analysis (IGA) \Bl aims at bridging the gap \B between computer aided geometric design and finite element analysis \cite{cottrell_isogeometric_2009,hughes_isogeometric_2005}. The core idea of IGA is to use the same \textit{smooth} and \textit{higher-order} spline basis functions for the representation of both geometry and the approximation of physics-based field solutions. For a more detailed introduction to IGA, we refer interested readers to the recent review articles \cite{Haberleitner:17.1,Hughes:17.1,marussig2018review,Schillinger:18.1} and the references \Bl cited \B therein. Due to the higher-order nature of splines, IGA is particularly attractive for higher-order accurate structural analysis. For elastostatic-type problems, higher-order IGA is superior to standard FEA in terms of per-degree-of-freedom accuracy, as shown e.g.\ in \cite{Schillinger:13.2}.

A key property of higher-order IGA, already discussed in one of the first articles \cite{cottrell_isogeometric_2006}, is its well-behaved discrete spectrum of eigenfrequencies and eigenmodes. The associated potential of IGA for efficient explicit dynamics, however, lies largely idle to this day.  One of the main barriers is the lack of a diagonalization scheme for the mass matrix that enables a computationally and memory efficient update, but also maintains the higher-order spatial accuracy of the method \cite{cottrell_isogeometric_2006,hartmann2015mass}. For instance, Hartmann and Benson reported in \cite{hartmann2015mass} that ``increasing the degree on a fixed mesh size increases the cost without a commensurate increase in accuracy''. Therefore, studies on isogeometric explicit dynamics schemes that have adopted standard row-sum lumping to diagonalize the mass matrix have focused on lower-order quadratic spline discretizations \cite{Benson:10.1,chen2014explicit,leidinger2019explicit}. 

One idea to bridge this gap is to apply additional corrector updates that restore higher-order accuracy \cite{Evans_corrector_scheme_2018} at low memory requirements, but lead to an (unacceptable) increase of the computational cost. A recent study \cite{voet2023masslumping} explored ideas from preconditioning to develop certain customized mass matrices that approximate the consistent mass matrix while retaining some of the beneficial properties of lumped mass. However, the presented methods have not led to a higher order accurate technology. For standard $C^0$-continuous basis functions, a pathway to higher-order accurate explicit dynamics with diagonal mass matrices is the adoption of \Bl Lagrange basis functions collocated at the Gauss-Lobatto Legendre (GLL) nodes \B \cite{Canuto:07.1, duczek_lumping_sem2019,Schillinger2014,Willberg:12.1}. Due to the interpolatory property of nodal basis functions, this strategy naturally yields a consistent and higher-order accurate diagonal mass matrix. Another idea is therefore the transfer of this mechanism to isogeometric analysis via a change of basis in the context of Lagrange extraction \cite{nguyen2017collocated,schillinger2016lagrange}. But the associated matrix operations, even when accelerated through special techniques, increase the computational cost and the memory requirements beyond an acceptable level. 

Another approach to achieve a diagonal mass matrix is the adoption of dual spline functions. This set of functions is defined bi-orthogonal (or ``dual'') to a set of standard B-splines (or NURBS), such that the inner product of the two integrated over a domain yields a diagonal matrix. Hence, if a dual basis is used as the test space in a finite element formulation, we naturally obtain a consistent diagonal mass matrix \cite{Anitescu_dual_2019}. In CAD research, this bi-orthogonality property has been well known for decades, and several approaches for the construction of the underlying dual basis exist, see e.g.\ \cite{schumaker_spline_2007}. In the computational mechanics community, this property has been used for different purposes, e.g.\ for dual mortar methods 
\cite{Dornisch_dual_basis_2017, seitz_mortar_2016, zou_mortar_2018} or for unlocking of Reissner-Mindlin shell formulations \cite{zou_dual_locking_2020}.

Dual spline basis functions still suffer from shortcomings when applied as discrete test functions in elastodynamics as envisioned in \cite{Anitescu_dual_2019}, to the extent that their use in explicit dynamics seems not straightforward (as recently summarized in \cite{nguyen2023}): Firstly, bi-orthogonality in general holds only on the parametric domain, but is in general lost under non-affine geometric mapping. Secondly, dual spline functions with the same smoothness as their B-spline counterparts have global support on each patch, thus producing fully populated stiffness forms. Their support can be localized, but at the price of losing continuity, leading to discontinuous functions in the fully localized case. Both are prohibitive in a finite element context. And thirdly, dual basis functions are not interpolatory at the boundaries. Therefore, the identification of a kinematically admissible set of test functions that allows the variationally consistent strong imposition of Dirichlet boundary conditions is not straightforward.

In a recent paper \cite{nguyen2023}, we presented an isogeometric Petrov-Galerkin method for explicit dynamics that overcomes the first two of the three challenges. Its key ingredient is a class of ``approximate'' dual spline functions, introduced in \cite{chui2004nonstationary} and successfully applied within IGA for dual mortaring in \cite{Dornisch_dual_basis_2017}. It only approximately satisfies the discrete bi-orthogonality property, but preserves all other properties of the original B-spline basis, such as its smoothness, polynomial reproduction and local support. Employing the approximate dual functions to discretize the test function space, we obtain a semidiscrete Petrov-Galerkin formulation in the format \eqref{eq:explicit}, whose mass matrix $\boldsymbol{M}$ is not truly diagonal, but a ``close'' approximation of a diagonal matrix. We showed that we can then apply standard rowsum mass lumping to diagonalize the approximately diagonal mass matrix without compromising higher-order spatial accuracy. \Bl A similar study was recently presented in \cite{held2023efficient}, also with promising results. \B

In this paper, we build on and extend our results \Bl presented \B in \cite{nguyen2023}, with the goal of providing a comprehensive mathematical framework of approximate dual basis functions, including a consistent method to strongly incorporate Dirichlet boundary conditions. In Section \ref{sec:background}, we set the stage by fixing a general linear second order model problem, which we discretize in space via the Galerkin method and in time via higher-order explicit time integration. We review different classical ways of performing mass lumping. In Section \ref{sec:approxduals}, we focus on a univariate basis for splines that is approximately dual to B-splines in the $L^2$-norm and discuss their properties. We demonstrate that the Petrov-Galerkin method with rowsum mass lumping proposed in \cite{nguyen2023} can also be interpreted as a Galerkin method with a customized mass matrix. Based on this observation, we show that the customized mass technique does not have a negative effect on the asymptotic accuracy of the method. Another important extension with respect to our work in \cite{nguyen2023} is the incorporation of Dirichlet boundary conditions. We show that polynomial accuracy is maintained under strong imposition of boundary conditions. In Section \ref{sec:multivariate}, we extend our developments to the multivariate setting, explain how the properties of the basis can be preserved under general non-affine mappings, and discuss key aspects of an efficient computer implementation. In Section \ref{sec:results}, we verify accuracy and robustness of our methodology in the context of spectral analysis and a two-dimensional linear explicit analysis benchmark, demonstrating that our Petrov-Galerkin scheme indeed preserves higher-order accuracy in explicit dynamics calculations. In Section \ref{sec:conclusion}, we provide a summary and outline potential future research.
\section{Preliminaries  \label{sec:background}}
In this section we present background and notation used in the sequel of this paper. First, we present a general linear second order model problem. The Galerkin method is used to discretize the equations in space and we briefly discuss explicit time integration. Subsequently, we discuss the consistent mass matrix and different classical ways of performing mass lumping.

\subsection{Model problem}
Let $\Domain \subset \mathbb{R}^{\nsd}$ be an open set with piecewise smooth boundary $\Gamma = \overline{\Gamma_{g} \cup \Gamma_{h}}, \; \Gamma_{g} \cap \Gamma_{h} = \emptyset$. Consider (sufficiently smooth) prescribed boundary and initial data 
\begin{subequations}
\begin{align}
	f \; &: \; \Domain \times (0,T)\mapsto \mathbb{R}, \\
	g \; &: \; \Gamma_{g} \times (0,T) \mapsto \mathbb{R}, \\
	h \; &: \; \Gamma_{h} \times (0,T) \mapsto \mathbb{R}, \\
	u_{0}, \, \dot{u}_{0} \; &: \; \Domain \mapsto \mathbb{R}.
\end{align}
Let $\funspace \equiv H^1(\Domain)$. The space of trial solutions is assumed to be constant in time,
\begin{align}
	\funspace_g &:= \left\{v \in \funspace \text{ with } v(x) = g(x) \; \text{for } x \in \Gamma_{g}   \right\}.
\end{align}
Here, the Dirichlet data, $g$, is built directly into the trial-space, while the Neumann data, \Bl $h$, \B is incorporated weakly through natural imposition.

We consider weak formulations of the following form: given $f, \, g, \, h, u_{0}$ and $\dot{u}_{0}$, find $u(t) \in \funspace_g$ such that for all $w \in \funspace_0$
\begin{align}
	\bilinearform{a}{u}{w} + \bilinearform{b}{\ddot{u}}{w} &= l(w),	\\
	\bilinearform{b}{u(0)}{w} &= \bilinearform{b}{u_0}{w},	\\
	\bilinearform{b}{\dot{u}(0)}{w} &= \bilinearform{b}{\dot{u}_0}{w}.
\end{align}
Here, $a \; : \; \funspace \times \funspace \mapsto \mathbb{R}$ is a positive semi-definite and symmetric bilinear form (definite on the subspace $\funspace_0 \times \funspace_0$) and $b \; : \; \funspace \times \funspace \mapsto \mathbb{R}$ is a positive definite symmetric bilinear form defined as: $b(u,v) =  \int_{\Domain} \rho \, u \, v \, d \Domain$, where $\rho \, : \, \Domain \mapsto \mathbb{R}$ denotes the density of the medium. Furthermore, $l \; : \funspace \mapsto \mathbb{R}$ is a time dependent linear form dependent on the prescribed forces and Neumann boundary data: $l(w) = \int_{\Domain} w \, f  \, d \Domain + \int_{\Gamma_h} w \, h \, d \Gamma$. 
\end{subequations}

The model problem described here can represent different evolutionary processes, such as the time-dependent linear heat equation $(\nsd=1,2,3)$, vibration of a string ($\nsd = 1$), or vibration of a linear membrane ($\nsd = 2$).

\subsection{Semi-discrete equations of motion}
We consider semi-discrete systems of equations obtained using the Galerkin method with a finite dimensional space $\funspace^h \subset \funspace$ and subspace $\funspace^h_0 \subset \funspace_0$ discretized in terms of tensor product spline functions. Let $u^h = v^h + g^h$ where $v^h \in \funspace^h_0$ and $g^h \in \funspace^h$ approximately satisfies the Dirichlet boundary data $g$ on $\Gamma_g$.

The Galerkin formulation reads: find $v^h(t) \in \funspace^h_0$ such that for all $w^h \in \funspace^h_0$
\begin{subequations}
\begin{align}
	\bilinearform{a}{w^h}{v^h} + \bilinearform{b}{w^h}{\ddot{v}^h} &= l(w^h) - \bilinearform{a}{w^h }{g^h} - \bilinearform{b}{w^h}{\ddot{g}^h},	\\
	\bilinearform{b}{w^h}{v^h(0)} &= \bilinearform{b}{w^h}{u_0} - \bilinearform{b}{w^h}{g^h(0)},	\\
	\bilinearform{b}{w^h}{\dot{v}^h(0)} &= \bilinearform{b}{w^h}{\dot{u}_0} - \bilinearform{b}{w^h}{\dot{g}^h(0)}.
\end{align}
\label{eq:galerkin}
\end{subequations}
Consider finite dimensional representations in terms of basis functions $\trialfun{i}(x) \in \funspace^h$
\begin{align}
	v^h(x) = \sum_{\idx{i} \in \eta - \eta_g} \trialfun{i}(x) \,  d_{\idx{i}}(t), \qquad g^h(x) = \sum_{\idx{i} \in \eta_g} \trialfun{i}(x) \, d_\idx{i}(t),
\end{align} 
where $\eta$ is an index set of all functions in $\funspace^h$ and $\eta_g$ is an index set of functions that intersect the boundary. We consider matrices $\mmat{M} = [M_{\idx{ij}}]$ and $\mmat{K} = [K_{\idx{ij}}]$, and vectors $\mmat{F} = [F_{\idx{i}}]$, $\tilde{\vect{d}}_0 = [\tilde{d}_{0\idx{i}}]$, and $\dot{\tilde{\vect{d}}}_0 = [\dot{\tilde{d}}_{0\idx{i}}]$ ($\idx{i}, \idx{j} \in \eta - \eta_g$) defined by
\begin{subequations}
\begin{align}
	M_{\idx{ij}} &:= \bilinearform{b}{\testfun{i}}{\trialfun{j}},	\\
	K_{\idx{ij}} &:= \bilinearform{a}{\testfun{i}}{\trialfun{j}}, \\
	F_{\idx{i}} &:= l(\testfun{i}) - \bilinearform{a}{\testfun{i}}{g^h} - \bilinearform{b}{\testfun{i}}{\ddot{g}^h}, \\
	\tilde{d}_{0\idx{i}} &:= b(\testfun{i}, u_0 - g^h(0)), \\
	\dot{\tilde{d}}_{0\idx{i}} &:= b(\testfun{i}, \dot{u}_0 - \dot{g}^h(0)).
\end{align}
\end{subequations}
Applying these definitions to \eqref{eq:galerkin} leads to the following matrix problem,
\begin{subequations}
\begin{align}
	\mmat{M} \, \vect{\ddot{d}}(t) + \mmat{K} \, \vect{d}(t) &= \vect{F}(t),  \quad t \in (0, T), 	\label{eq:update}\\
	\mmat{M} \, \vect{d}(0) &= \tilde{\vect{d}}_0,	\\
	\mmat{M} \, \vect{\dot{d}}(0) &= \dot{\tilde{\vect{d}}}_0.
\end{align}
\end{subequations}
This is a coupled system of ordinary differential equations (ODE's) for the displacement coefficients $\vect{d}(t):=[d_{\idx{i}}(t)]$, where $\vect{\dot{d}}(t):=[\dot{d}_{\idx{i}}(t)]$ denote velocities and $\vect{\ddot{d}}(t):=[\ddot{d}_{\idx{i}}(t)]$ the accelerations.

The properties of the stiffness matrix, $\mmat{K}$, and the consistent mass matrix, $\mmat{M}$, follow directly from the bilinear forms $a$ and $b$, respectively, and the finite element basis functions $\List{\trialfun{i}, \; \idx{i} \in \eta - \eta_g }$. In particular, the consistent mass matrix $\mmat{M}$ and the stiffness matrix $\mmat{K}$ are both positive definite and symmetric and have a sparse, banded structure due to the local support of the basis functions. In addition, the positivity of B-splines as trial and test functions leads to non-negative entries in the mass matrix.

\subsection{Explicit time stepping methods}
Let $\vect{d}_k, \, \dot{\vect{d}}_{k},$ and $\ddot{\vect{d}}_{k}$ denote certain finite difference approximations of $\vect{d}(t_k)$, $\dot{\vect{d}}(t_{k})$ and $\ddot{\vect{d}}(t_{k})$, respectively, that describe the state at time $t_k \in (0,T)$, and let $\vect{F}_{k} = \vect{F}(t_{k})$. Given the state at a particular time $t_n \in [0,T)$, explicit methods determine the state at a later time $t_{n+1} \in [0,T]$ by replacing the update in \eqref{eq:update} with the matrix problem
\begin{align}
	\mmat{M} \, \ddot{\vect{d}}_{n+1} = \vect{F}_{n+1} - \mmat{K} \, \vect{d}_n.
\end{align}
The stiffness matrix $\mmat{K}$ need not be explicitly formed. Instead, its action on
a vector $\vect{d}_n$ may be implemented in a matrix-free context. A solution update using the consistent mass matrix $\mmat{M}$ still requires a direct matrix solve, which scales $\mathcal{O}(N^2)$ where $N$ is the leading dimension of the system.

\subsection{Efficient update through mass lumping}
Mass lumping refers to a procedure where the consistent mass matrix is somehow approximated by a diagonal matrix of the same dimensions to enable fast explicit solution updates, that is, updates that do not require matrix inversion and scale $\mathcal{O}(N)$. The three important procedures are (see e.g. \cite[Section 12.2.4]{zienkiewicz2013finite}):
\begin{enumerate}
	\item The rowsum method: $M^{\text{lumped}}_{\idx{ii}} := \sum \nolimits_{\idx{j}} M_{\idx{ij}}$.
	\item Diagonal scaling: $M^{\text{lumped}}_{\idx{ii}} = c M_{\idx{ii}}$ with $c$ chosen such that $\sum \nolimits_\idx{j} M^{\text{lumped}}_{\idx{jj}} = \sum \nolimits_\idx{i} \sum \nolimits_{\idx{j}} M_{\idx{ij}} = \int_{\Omega} \rho d \Omega$.
	\item Evaluation of $M_{\idx{ij}}$ using a nodal quadrature rule corresponding to nodal shape functions.
\end{enumerate}
These techniques have proven robust and efficient mainly for linear finite element methods. The third procedure has had success in higher order spectral element methods, see \cite{Schillinger2014}. In this work we use rowsum lumping and show that higher order accuracy can be maintained for a special class of lumped mass matrices in which the test and trial functions are not the same.
\section{Approximate \Bl $L^2$ dual \B bases \label{sec:approxduals}}
In this section we introduce a basis for splines that is approximately dual to B-splines in the $L^2$-norm. In contrast to the exact dual functions, these functions are locally supported, which makes them suitable as test functions in a Petrov-Galerkin method. The method we propose can be interpreted as: (1) a Galerkin method with a customized mass matrix; \Bl or \B (2) a Petrov-Galerkin method with rowsum mass lumping. We show that these perspectives are equivalent and \Bl demonstrate that \B the customized mass / mass-lumping technique does not have a negative effect on the asymptotic accuracy of the method. An important extension with respect to previous work is the \Bl consistent \B incorporation of Dirichlet boundary conditions. We show that polynomial accuracy is maintained under strong imposition of boundary conditions.

\subsection{Univariate B-splines}
A spline is a piecewise polynomial that is characterized by the polynomial degree of its segments and the regularity prescribed at their interfaces. A convenient basis in which to represent polynomial splines is given by
B-splines  (see e.g. \cite{boor_practical_2001}). Consider a partitioning $\domain$ of the univariate interval $[a,b]$ into $\nelms$ elements
\begin{align}
	a = \bp{0} < \bp{1} < \ldots < \bp{\nelms} = b.
\end{align}
With every internal \emph{breakpoint}, $\bp{k}$, we may associate an integer, $\r_k$, that prescribes the smoothness between the polynomial pieces. Then, given the \emph{knot}-multiplicity, $\List{\p-\r_k}_{k=1}^{\nelms-1}$, we can define the \emph{knot vector},
\begin{align}
	\knots = \List{\knot{i}}_{i=1}^{\ndofs+\p+1} := 
	\left\{  \right. 
		\underbrace{\bp{0}, \ldots , \bp{0}}_{\p+1}, \ldots ,
		\underbrace{\bp{k-1}, \ldots , \bp{k-1}}_{\p-\r_{k-1}}, 
		\underbrace{\bp{k}, \ldots , \bp{k}}_{\p-\r_{k}}, \ldots,
		\underbrace{\bp{\nelms}, \ldots , \bp{\nelms}}_{\p+1}
		\left. \right\}
\end{align}
With a knot vector in hand, B-splines are stably and efficiently computed using the Cox–de Boor recursion,
\begin{align}
	\bspline{i,0}(\param) &= 
	\begin{cases} 
		1 \quad \text{if } \param \in [\knot{i}, \knot{i+1}) \\
		0 \quad \text{otherwise}
	\end{cases} \\
	\bspline{i,\p}(\param) &= \frac{\param - \knot{i}}{\knot{i+\p} - \knot{i}} \bspline{i,\p-1}(\param) 
	+ \frac{\knot{i+\p+1} - \param}{\knot{i+\p+1} - \knot{i+1}} \bspline{i+1,\p-1}(\param) 
\end{align}
where, by definition, $0/0 = 0$.

Let $\r = \List{0 \leq \r_{k} \leq \p-1, \; k = 1, \ldots , \nelms-1}$. The space of polynomial splines of degree $\p$ and dimension $\ndofs$ with $\r_k$ continuous derivatives at breakpoint $\bp{k}$ is defined as
\begin{align}
	\sspace{\p}{\r} (\domain) := \Span{\bspline{i,\p} (\param), \; i=1, \ldots, \ndofs}.
\end{align}
To shorten notation we drop the dependence on the domain and simply write $\sspace{\p}{\r}$.

B-splines have important mathematical properties, many of which are useful in design as well as in analysis. B-spline basis functions of degree $\p$ may have up to $\p-1$ continuous derivatives, they form a positive partition of unity, and have local support of up to $\p+1$ elements.

\subsection{$L^2$ dual basis}
Let $\pairing{\cdot}{\cdot} \, : \; L^2(\domain) \times L^2(\domain) \mapsto \mathbb{R}$ denote the standard $L^2$ inner product of two square integrable functions. A set of functions $\List{\dual{i}, \; i=1, \ldots , \ndofs}$ is called a dual basis (in the $L^2$-norm) to the set $\List{\bspline{i}, \; i=1, \ldots , \ndofs}$ if the following properly holds
\begin{align}
	\pairing{\dual{i}}{\bspline{j}} = \delta_{ij}, \quad i,j \in 1, \ldots , \ndofs.
	\label{eq:duality}
\end{align}
Here, $\delta_{ij}$ denotes the Kronecker delta, which is one if $i=j$ and zero otherwise. The one-to-one correspondence implies that the dual functions are linearly independent if and only if the functions $\List{\bspline{i}, \; i=1, \ldots , \ndofs}$ are linearly independent.

A dual basis is not unique. There is freedom in choosing a space spanned by the functions, there are different ways to construct them, and their properties, such as smoothness and support, may vary. If the dual basis functions span the space of splines $\sspace{\p}{\r}$, then they are uniquely defined and can be computed as (duality in  \eqref{eq:duality} may be checked by direct substitution)
\begin{align}
	\dual{i}(\param) = \sum_{j=1}^{\ndofs} \left(\mat{\grammian}^{-1}\right)_{ij} \bspline{j}(\param), \quad \text{where } \grammian_{ij} = \pairing{\bspline{i}}{\bspline{j}}.
	\label{eq:dualbasis}
\end{align}

A Galerkin discretization employing $L^2$ dual functions as test functions and B-splines as trial functions yields the identity matrix as mass matrix. Although this property is highly sought after in explicit dynamics, they are unsuitable as a basis for the test-space. Because the inverse of the Grammian matrix, $\mat{\grammian}^{-1}$, is a dense matrix, the $L^2$ dual functions have global support. This leads to poor scaling and expensive formation and assembly of the interior force vector in explicit dynamics computations, because each entry in the vector depends on all the quadrature points.

Finally, we list a well known property of the dual basis used in the context of $L^2$-projection (Proposition \ref{eq:projection}). In the next subsection we introduce a set of functions that are (in some sense) approximately dual to B-splines, yet maintain local support, and satisfy an approximation property similar to the following.
\begin{proposition}[$L^2$-projection\label{eq:projection}] Let $f \in L^2(\domain)$. The function $u(\param) = \sum_{i=1}^{\ndofs} \pairing{f}{\dual{i}} \, \bspline{i}(\param)$ is the best spline approximation to $f$ in the $L^2$-norm. In other words, it is the minimizer of the convex minimization problem
\begin{align*}
	\underset{u \in \sspace{\p}{\r}(\domain)}{\text{argmin}} \; \tfrac{1}{2} \pairing{u}{u} - \pairing{u}{f}
\end{align*}
\end{proposition}
\begin{proof} This follows by direct application of the dual functions.
\end{proof}

\subsection{$L^2$ approximate dual basis}
We loosen the restriction of duality and search for a set of functions $\List{\approxdual{i}, \; i=1, \ldots , \ndofs}$ with the following properties:
\begin{subequations}
\begin{enumerate}
	\item A basis for the spline space $\sspace{\p}{\r}$.
	\item Approximate $L^2$ duality with B-splines.
	\item Reproduction of all polynomials of degree $< \p+1$.
	\item Local compact support.
\end{enumerate}
\end{subequations}

The construction relies on symmetric positive definite (SPD) matrices. With $\mathbb{SPD}^{\ndofs} \subset \mathbb{R}^{\ndofs \times \ndofs}$ we denote the set of all such matrices of $\ndofs$ rows and columns. We recall that SPD matrices are invertible and their inverse is also SPD.

\begin{definition}[Approximate dual basis] Let $\pspace{\p}(\domain) \subset \sspace{\p}{\r}(\domain)$ denote the set of polynomials of degree $< p+1$ and let $\mat{S} \in \mathbb{SPD}^{\ndofs}$. We call the set of functions
\begin{align}
	\approxdual{i}(\param) = \sum_{j=1}^{\ndofs} S_{ij} \, \bspline{j}(\param), \quad i = 1, 2, \ldots , \ndofs,
	\label{eq:approxdualbasis}
\end{align}
an approximate dual basis to the set of B-splines if
\begin{align}
	\pairing{f}{\approxdual{i}} = \pairing{f}{\dual{i}}  \quad \forall f \in \pspace{\p}(\domain).
	\label{eq:approxduality}
\end{align}
\end{definition}
The identity \eqref{eq:approxduality} implies approximate duality with B-splines. Therefore, the functions $\approxdual{i}, \; (i=1, \ldots , \ndofs)$ will be referred to as \emph{approximate dual functions}. Their collection forms a basis for $\sspace{\p}{\r}$  (since $\mat{S}$ is invertible) and is called an \emph{approximate dual basis}. We note that the representation is not unique. A \Bl particularly \B useful approximate dual basis has been developed in \cite{chui2004nonstationary}. The approximate dual functions presented therein have minimum compact support, in the sense that there exists no other basis that satisfies the approximate duality with smaller supports. In the remainder of this paper we shall base our construction on that approach.

SPD matrices induce a discrete inner product on finite dimensional spaces. 
\begin{definition}[Discrete inner product on splines] Let $u, v \in \sspace{\p}{\r}(\domain)$ be expanded in terms of B-splines and coefficients as $u(\param) = \sum_{i=1}^{\ndofs} u_i \, \bspline{i}(\param)$ and $v(\param) = \sum_{i=1}^{\ndofs} v_i \, \bspline{i}(\param)$, respectively. An SPD-matrix $\mat{\hat{\grammian}} \in \SPD{\ndofs}$ induces a discrete inner product $\pairing{u}{v}_{\hat{\grammian}} \; : \; \sspace{\p}{\r} \times \sspace{\p}{\r} \mapsto \mathbb{R}$ defined by
\begin{align}
	\pairing{u}{v}_{\hat{\grammian}} := \sum_{i=1}^{\ndofs} \sum_{j=1}^{\ndofs} u_i \, \hat{\grammian}_{ij} \, v_j.
\end{align}
\end{definition}
We note that $\pairing{u}{v}_{\grammian}$ (where $\mat{\grammian}$ is the Grammian matrix) coincides with the usual $L^2$ inner product on splines, that is, $\pairing{u}{v}_{\grammian} = \pairing{u}{v}$ for spline functions $u$ and $v$. Similar to \eqref{eq:dualbasis} the approximate dual functions in \eqref{eq:approxdualbasis} are expressed as linear combinations of B-splines via a matrix vector product involving an SPD matrix. Hence, analogous to Lemma \eqref{eq:projection}, the associated quasi-interpolant may be associated to a minimization problem.

\begin{theorem}[Quasi $L^2$-projection \label{th:l2quasiprojection}] Let $f \in L^2(\domain)$ and $\mat{\hat{\grammian}} = \mat{S}^{-1}$. The spline function $u(\param) = \sum_{i=1}^{\ndofs} \pairing{f}{\approxdual{i}} \, \bspline{i}(\param)$ is the unique solution to the convex minimization problem
\begin{align}
	\underset{u \in \sspace{\p}{\r}(\domain)}{\text{argmin}} \; \tfrac{1}{2}\pairing{u}{u}_{\hat{\grammian}} - \pairing{u}{f}.
	\label{eq:minimizer1}
\end{align}
Furthermore, $u = f$ whenever $f \in \mathbb{P}^p$.
\end{theorem}
\begin{proof} Let $u(\param) =  \sum_{i=1}^{\ndofs} u_i \, \bspline{i}(\param)$. The optimization problem is equivalent to the following set of linear equations
\begin{align}
	\sum_{j=1}^{\ndofs}  \hat{\grammian}_{ij} \, u_j = \pairing{f}{\bspline{i}}, \quad i = 1, \ldots , \ndofs.
\end{align}
Its solution is $u_i =  \sum_{j=1}^{\ndofs} S_{ij} \, \pairing{f}{\bspline{j}} =  \pairing{f}{\approxdual{i}}$. If $f \in \mathbb{P}^p$, then $\pairing{f}{\approxdual{i}} = \pairing{f}{\dual{i}}$. Here, $u_i = \pairing{f}{\dual{i}}$ is the best approximation in the $L^2$-norm. It reproduces all splines in the space, which includes all polynomials $f \in \mathbb{P}^p$.
\end{proof}

\subsection{Treatment of Dirichlet boundary conditions}
By associating the quasiinterpolant with the solution to a certain convex minimization problem it becomes straightforward to consider linear constraints, such as Dirichlet boundary conditions. Without loss of generality we focus on the application of a single homogeneous boundary condition on the left boundary of the univariate domain. 

\begin{theorem}[Quasi $L^2$-projection with a homogeneous boundary condition] Suppose $f \in H^1(\domain)$ satisfies $f(a) = 0$. The spline function $u(\param) = \sum_{i=1}^{\ndofs} \pairing{f}{\approxdual{i}} \, \bspline{i}(\param)$, where
\begin{align}
\label{eq:dual2}
	\begin{bmatrix}
		\approxdual{2}(\param) \\
		\vdots \\
		\approxdual{\ndofs}(\param)
	\end{bmatrix}
	 = 
	\begin{bmatrix} 
		\hat{\grammian}_{22} 		& \cdots & \hat{\grammian}_{2\ndofs} 					\\
		\vdots					& \ddots & \vdots 							\\
		\hat{\grammian}_{\ndofs2} 	& \cdots & \hat{\grammian}_{\ndofs \ndofs}
	\end{bmatrix}^{-1}
	\begin{bmatrix}
		\bspline{2}(\param) \\
		\vdots \\
		\bspline{\ndofs}(\param)
	\end{bmatrix},
\end{align}
is the unique solution to the convex minimization problem
\begin{align}
	\underset{u \in \sspace{\p}{\r}}{\text{argmin}} \; \tfrac{1}{2} \pairing{u}{u}_{\hat{\grammian}} - \pairing{u}{f} \quad \text{subject to } u(a) = 0.
	\label{eq:minimizer2}
\end{align}
Furthermore, if $f \in \mathbb{P}^p$, then $u = f$.
\end{theorem}
\begin{proof}
In matrix-notation we may write (making use of $u_1 = u(a) = 0$)
\begin{align}
	\underset{u_i \in \mathbb{R}}{\text{argmin}} \; \tfrac{1}{2}\sum_{i=2}^{\ndofs} \sum_{j=2}^{\ndofs} u_i \, \hat{\grammian}_{ij} \, u_j \; - \sum_{i=2}^{\ndofs} u_i \, \pairing{f}{\bspline{i}},
	\label{eq:minimizer3}
\end{align}
which is equivalent to the matrix problem
\begin{align*}
	\sum_{j=2}^{\ndofs} \hat{\grammian}_{ij} \, u_j  = \pairing{f}{\bspline{i}}, \quad i = 2, \ldots , \ndofs.
\end{align*}
Its solution can again be presented in terms of approximate dual functions: $u_i  = \pairing{f}{\approxdual{i}}$, where the linear combination is goverened by \eqref{eq:dual2}. 
 
The constraint does not alter the energy in \eqref{eq:minimizer2} (which is obvious from the expression in \eqref{eq:minimizer3}). Hence, a solution of \eqref{eq:minimizer2} is a solution to \eqref{eq:minimizer1} if one considers only those functions that are exactly reproduced and satisfy the constraint: polynomials $f \in \mathbb{P}^p$ that satisfy $f(a) = 0$. Hence, the quasiinterpolant preserves all polynomials that satisfy the boundary condition.
\end{proof}

We note that it is not necessary to compute the inverse of the submatrix of $\mat{\hat{\grammian}}$ directly. One can make efficient use of the Woodbury matrix identity \cite[Section 2.1.4]{golub2013matrix} to eliminate any unnecessary computations of inverse matrices
\begin{align}
	\left(\mat{\hat{\grammian}} + \mat{U} \mat{V}^T \right)^{-1} = \mat{S} - \mat{S} \mat{U}  \left(\mat{I}_2  + \mat{V}^T \mat{S} \mat{U}  \right)^{-1} \mat{V}^T \mat{S}.
\end{align}
In our case, the matrix $\mat{\hat{\grammian}} + \mat{U} \mat{V}^T$ is a rank two update of the original matrix
\begin{align*} 
	\underbrace{\begin{bmatrix} 
		\hat{\grammian}_{11}	& 0 						& \cdots & 0 									\\
		0 				& \hat{\grammian}_{22} 			& \cdots & \hat{\grammian}_{2\ndofs} 				\\
		\vdots			& \vdots					& \ddots & \vdots 							\\
		0				&  \hat{\grammian}_{\ndofs2} 	& \cdots & \hat{\grammian}_{\ndofs \ndofs}
	\end{bmatrix}}_{\mat{\hat{\grammian}} + \mat{U} \mat{V}^T}
	=
	\underbrace{\begin{bmatrix} 
		\hat{\grammian}_{11}			& \hat{\grammian}_{12} 			& \cdots & \hat{\grammian}_{1\ndofs} 		\\
		\hat{\grammian}_{21} 		& \hat{\grammian}_{22} 			& \cdots & \hat{\grammian}_{2\ndofs} 		\\
		\vdots					& \vdots					& \ddots & \vdots 					\\
		\hat{\grammian}_{\ndofs1}	& \hat{\grammian}_{\ndofs2} 	& \cdots & \hat{\grammian}_{\ndofs \ndofs}
	\end{bmatrix}}_{\mat{\hat{\grammian}}}
	+
	\underbrace{\begin{bmatrix} 
		0						& -1 			\\
		\hat{\grammian}_{21} 		& 0 			\\
		\vdots					& \vdots		\\
		\hat{\grammian}_{\ndofs1}	& 0
	\end{bmatrix}}_{\mat{U}}
	\underbrace{\begin{bmatrix} 
		-1		& 0 				& \cdots & 0 		\\
		0 		& \hat{\grammian}_{21} 	& \cdots & \hat{\grammian}_{\ndofs1}
	\end{bmatrix}}_{\mat{V}^T}
\end{align*}

\subsection{Connection with rowsum mass lumping}
We briefly discuss the connection with rowsum mass lumping. As in \cite{nguyen2023} one can use approximate dual functions directly as test functions in a Petrov-Galerkin method. The resulting mass matrix is
\begin{align}
	C_{ij} := \pairing{\approxdual{i}}{\bspline{j}} = S_{ik} \, \grammian_{kj}.
\end{align}
The equation above concerns the one-dimensional case with unit density. However, many properties carry over to the multidimensional setting using \Bl Kronecker products and \B the techniques presented in the next section (see also  \cite{nguyen2023}).

\begin{lemma} $\sum_{j=1}^{\ndofs} \pairing{\approxdual{i}}{\bspline{j}} = 1, \quad i=1,\ldots , \ndofs.$
\label{lemma:1}
\end{lemma}
\begin{proof} Partition of unity implies that $\sum_{j=1}^{\ndofs} \pairing{\approxdual{i}}{\bspline{j}} = \pairing{\approxdual{i}}{1}$. Next, consider reproduction of constants, $\sum_i \pairing{\approxdual{i}}{1} \, \bspline{i}(\param) = 1$. Linear independence of B-splines and partition of unity implies the result.
\end{proof}

\begin{corollary} $C^{\mathrm{lumped}}_{ij} = \delta_{ij}, \quad i=1,\ldots , \ndofs$.
\end{corollary}

Consider $\mat{\grammian} \, \vect{u} = \vect{f}$. It follows that $\mat{C} \, \vect{u} = \mat{S}\, \mat{\grammian} \, \vect{u} = \mat{S}\vect{f}$. Rowsum lumping matrix $\mat{C}$ implies that $\vect{u} = \mat{S}\vect{f}$, which is equivalent to $\mat{\hat{\grammian}} \, \vect{u} = \vect{f}$. So, the perspective of the Petrov-Galerkin method in combination with rowsum mass lumping is equivalent to our viewpoint that involves a standard Galerkin method with a customized mass matrix $\mat{\hat{\grammian}}$ that happens to have an explicit sparse inverse denoted by $\mat{S}$. \Bl A similar viewpoint was presented recently in \cite{held2023efficient}. \B

\section{Higher order accurate mass lumping \label{sec:multivariate}}
The approximate dual basis introduced in the previous section maintains its properties, in particular polynomial reproduction, under affine mappings only. Here we extend the theoretical results to the two-dimensional setting and incorporate non-linear geometric mappings as well.

\subsection{Extension to arbitrary isoparametric mappings}
Let $\Phi \; : \; \ParamDomain \rightarrow \Domain$ denote a mapping from the unitsquare to the physical domain and let $F(\x) = \nabla \Phi (\x)$ denote the Jacobian matrix at a point $\x \in \ParamDomain$. As usual, $\phi$ is an invertible mapping such that its Jacobian determinant is a positive function $\det{(F)}>0$. We let $c \circ \Phi = \det{(F)} \cdot \rho \circ \Phi$, which is a positive function by virtue of $\det{(F)}>0$ and $\rho >0$.

We proceed as in \cite{nguyen2023} and define trial and test functions as follows
\begin{subequations}
\begin{align}
	\trialfun{i} \circ \Phi &:= \bspline{i_1}(\param_1) \cdot \bspline{i_2}(\param_2), \quad \idx{i} \in \eta,  && (\text{trial functions})\\ 
	\dualtestfun{i} \circ \Phi &:= \frac{\approxdual{i_1}(\param_1) \cdot \approxdual{i_2}(\param_2)}{c \circ \Phi}, \quad \idx{i} \in \eta. && (\text{test functions}) \label{eq:test}
\end{align}
\end{subequations}
The entries of the consistent mass and stiffness \Bl matrices, \B the right hand side vector, and relevant boundary contributions become
\begin{subequations}
\begin{align}
	\mmat{M}_{\idx{ij}} &= \bilinearform{b}{\dualtestfun{i}}{\trialfun{j}} =  \sum_{\idx{k}} \mmat{S}_{\idx{i} \idx{k}} \; \bilinearform{b}{\testfun{k} / c}{\trialfun{j}}, \\
	\mmat{K}_{\idx{ij}} 	&= \bilinearform{a}{\dualtestfun{i}}{\trialfun{j}} = \sum_{\idx{k}}  \mmat{S}_{\idx{i} \idx{k}} \; \bilinearform{a}{\testfun{k} / c}{\trialfun{j}}, \\
	\mat{F}_{\idx{i}}		&= l(\dualtestfun{i}) - \bilinearform{a}{\dualtestfun{i}}{g^h} - \bilinearform{b}{\dualtestfun{i}}{\ddot{g}^h} = \sum_{\idx{k}}  \mmat{S}_{\idx{i} \idx{k}} \; \left( l(\testfun{k} / c) - \bilinearform{a}{\testfun{k} / c}{g^h} - \bilinearform{b}{\testfun{k} / c}{\ddot{g}^h} \right),	\\
	\tilde{d}_{0\idx{i}} &:= \bilinearform{b}{\dualtestfun{i}}{u_0 - g^h(0)} =  \sum_{\idx{k}}  \mmat{S}_{\idx{i} \idx{k}} \; \bilinearform{b}{\testfun{k} / c}{u_0 - g^h(0)},	\\
	\dot{\tilde{d}}_{0\idx{i}} &:= \bilinearform{b}{\dualtestfun{i}}{\dot{u}_0 - \dot{g}^h(0)} = \sum_{\idx{k}}  \mmat{S}_{\idx{i} \idx{k}} \; \bilinearform{b}{\testfun{k} / c}{\dot{u}_0 - \dot{g}^h(0)},
\end{align}
\end{subequations}
where $\mmat{S} = \mat{S}_2 \otimes \mat{S}_1$ is a Kronecker product matrix. Here, the subscripts refer to the component direction. Notably, the consistent mass matrix has a simple structure given by
\begin{align*}
	\mmat{M}_{\idx{ij}} &= \int_{0}^1 \approxdual{i_2}(\x_2) \bspline{j_2}(\param_2) \, d\param^2 \; \int_{0}^1 \approxdual{i_1}(\x_1) \bspline{j_1}(\param_1) d\param^1.
\end{align*}
Hence, the consistent mass matrix does not depend on the geometric mapping in any way since the function $c = \det{(F)} \cdot \rho$ in the denominator of the test functions \eqref{eq:test} cancels under the pullback to the reference domain. Alternatively, a
weighted inner product can be used to remove the dependency on the geometric mapping, see e.g. \cite{dornisch:2021}. 

It follows that all properties are maintained under geometric mappings. Moreover, the matrix can be factorized as follows:
\begin{align*}
	\mmat{M} &= \mat{S}_2 \mat{\grammian}_2 \otimes \mat{S}_1 \mat{\grammian}_1.
\end{align*}

\subsection{Higher order accuracy under mass lumping}
Prior to imposition of Dirichlet boundary conditions the semi-discrete equations associated with the explicit Petrov-Galerkin method are
\begin{align*}
	\mmat{M} \, \ddot{\vect{d}}_{n+1} = \mat{F}_{n+1} - \mmat{K} \, \vect{d}_n.
\end{align*}
Applying rowsum mass lumping ($\mmat{M}^{\text{lumped}} = \mmat{Id}$) the equations simplify to
\begin{align*}
	\ddot{\vect{d}}_{n+1} = \mat{F}_{n+1} - \mmat{K} \, \vect{d}_n.
\end{align*}
In the following we let
\begin{align}
	\mmat{\tilde{K}}_{\idx{ij}} 	&= \bilinearform{a}{\testfun{i} / c}{\trialfun{j}}, \\
	\mat{\tilde{F}}_{\idx{i}}		&= l(\testfun{i} / c) - \bilinearform{a}{\testfun{i} / c}{g^h} - \bilinearform{b}{\testfun{i} / c}{\ddot{g}^h}.
\end{align}
Hence, $\mmat{K} = \mmat{S} \, \mmat{\tilde{K}}$ and $\mat{F} = \mmat{S} \, \mat{\tilde{F}}$, where $\mmat{S} = \mat{S}_2 \otimes \mat{S}_1$ is a Kronecker product matrix. Using $\mmat{\hat{M}} = \mmat{S}^{-1} = \mat{S}^{-1}_2 \otimes \mat{S}^{-1}_1 = \mat{\hat{\grammian}}_2 \otimes \mat{\hat{\grammian}}_1$ we have:
\begin{align*}
	\mmat{\hat{M}} \, \ddot{\vect{d}}_{n+1} = \mat{\tilde{F}}_{n+1} - \mmat{\tilde{K}} \, \vect{d}_n.
\end{align*}
This is a direct extension of the $L^2$ quasi-projection (Theorem \ref{th:l2quasiprojection}) to the two-dimensional setting using the Kronecker product structure. The approximation properties in one dimension directly carry over to the multi-dimensional setting via the Kronecker product. Hence,  the approach maintains polynomial accuracy under mass lumping.

\subsection{Imposition of Dirichlet boundary conditions}
Dirichlet conditions are applied to the equations after performing mass lumping. We assume that the Dirichlet data are imposed separately for each side of the rectangular B-spline patch. Consequently, the mass matrix under strongly imposed homogeneous Dirichlet boundary conditions maintains a factorization $\underline{\mmat{\hat{M}}} = \underline{\mat{\hat{\grammian}}}_{2} \otimes \underline{\mat{\hat{\grammian}}}_{1}$, where homogeneous Dirichlet conditions are imposed on each of the univariate matrices $\underline{\mat{\hat{\grammian}}}_{k}, \; k=1,2$. If $\underline{\mmat{\hat{M}}}$, $\underline{\mmat{\tilde{K}}}$ and $\underline{\mat{\tilde{F}}}$ denote the matrices and vectors obtained after strong imposition of Dirichlet data, then the final semi-discrete equations are
\begin{align}
	\underline{\ddot{\vect{d}}}_{n+1} = \underline{\mmat{S}_2} \otimes \underline{\mmat{S}_1} \left(\underline{\mat{\tilde{F}}}_{n+1} - \underline{\mmat{\tilde{K}}} \, \underline{\vect{d}}_n \right).
	\label{eq:explicit_update}
\end{align}	
Again, the approximation properties carry directly over from the one-dimensional setting to the two-dimensional setting via the Kronecker product. Hence, the approach maintains polynomial precision, even with strongly imposed Dirichlet data.

\subsection{Implementational aspects}
We emphasize the following points with regards to the implementation:
\begin{itemize}
\item The Kronecker structure of $ \underline{\mmat{S}_2} \otimes \underline{\mmat{S}_1}$ facilitates efficient matrix storage that scales as $\propto 2(2\p+1) \cdot \sqrt{N}$ instead of $\propto (2\p+1)^2 \cdot N$. The payoff increases with polynomial degree and spatial dimension. In addition, matrix vector multiplications can be performed very efficiently since the matrices $\underline{\mat{S}}_k$ have a banded structure of contiguous non-zero entries.
\item The stiffness matrix need not be explicitly computed. Instead, the matrix vector product $\underline{\mmat{\tilde{K}}} \, \underline{\vect{d}}_n$ can be evaluated in a matrix-free fashion.
\item The entries of the stiffness matrix $\mmat{\tilde{K}}_{\idx{ij}} = \bilinearform{a}{\testfun{i} / c}{\trialfun{j}}$ are more complicated to evaluate than the entries to the standard Galerkin stiffness matrix due to the factor $1/c$ in the test function slot, effectively doubling the cost of evaluation for a second order problem and quadrupling it for a fourth order problem. 
\end{itemize}

\section{Numerical results \label{sec:results}}
In this section we verify accuracy and robustness of the proposed methodology. First, we analyze the discrete frequency spectrum obtained for a one-dimensional vibrating string that is clamped at either end. Here, we are particularly interested in the accuracy of the discrete frequencies obtained with the higher order approximated \Bl mass matrix \B as compared to the consistent and lumped \Bl mass matrices \B, specifically in the case of Dirichlet boundary conditions. Subsequently, we investigate the dynamics of a two-dimensional linear membrane. The benchmark problem verifies that higher order accuracy is maintained in the case of a nonlinear geometric mapping. We have combined our techniques with explicit Runge-Kutta time-stepping schemes of the appropriate order to maintain higher order accuracy in both space and time. Furthermore, the outlier removal procedures \Bl presented \B in \cite{Hiemstra_outlier_2021} are used to improve critical timestep values.

\subsection{Spectrum analysis}
The eigenvalue decomposition is an important tool in the study of the dynamic response of a system: it allows one to uncouple the equations such that each degree of freedom is essentially governed by its own uncoupled equation of motion (see e.g. \cite{hughes2012finite}). Each discrete eigenvalue is a measure of the amount of energy associated with one such response. Here we investigate the behavior of the discrete frequency spectra (the frequency is the square root of an eigenvalue) as a function of the normalized mode number for a one-dimensional second order problem\footnote{We refer to \cite{nguyen2023} for frequency spectra results associated with fourth order differential operators.}. The case studied here may be associated with the dynamics of a vibrating string that is fixed at either end. We investigate the spectra obtained for three different mass matrices: (1) consistent mass; (2) rowsum lumped mass; and (3) higher order approximate mass. Appropriate boundary constraints have been prescribed as explained in Section \ref{sec:approxduals} and outlier removal procedures have been applied as described in \cite{Hiemstra_outlier_2021}.

Figure \ref{fig:spectra1} displays the frequencies normalized by the corresponding analytical values. Hence, results closer to unity are better in terms of accuracy. It may be observed that a larger part of the spectrum is well approximated using the higher order approximate mass than with rowsum lumped mass. This is the case for all studied polynomial degrees, $p=2 - 5$. 
\begin{figure}[h]
	\centering
	\subfloat[$p=2$]{\includegraphics[trim = 0cm 0cm 0cm 0cm, clip,width=0.48\textwidth]{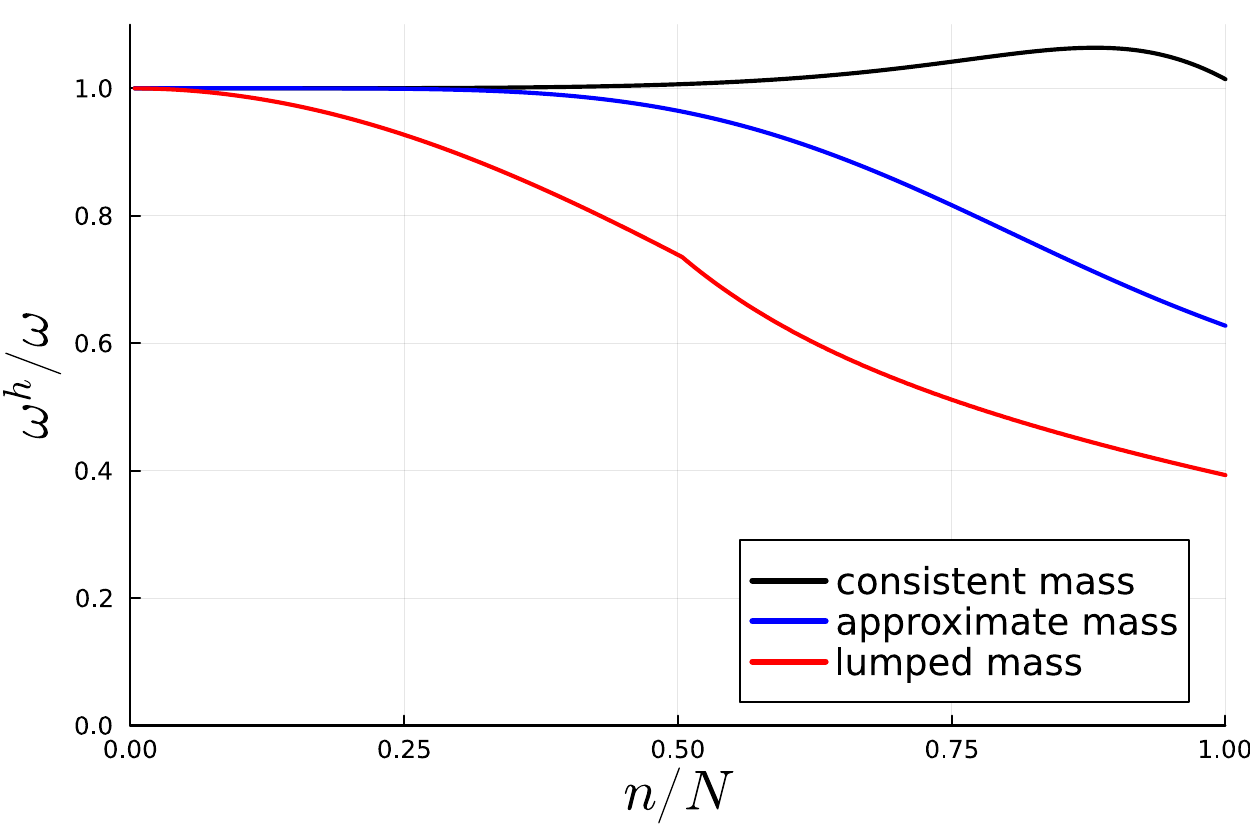}}
	\subfloat[$p=3$]{\includegraphics[trim = 0cm 0cm 0cm 0cm, clip,width=0.48\textwidth]{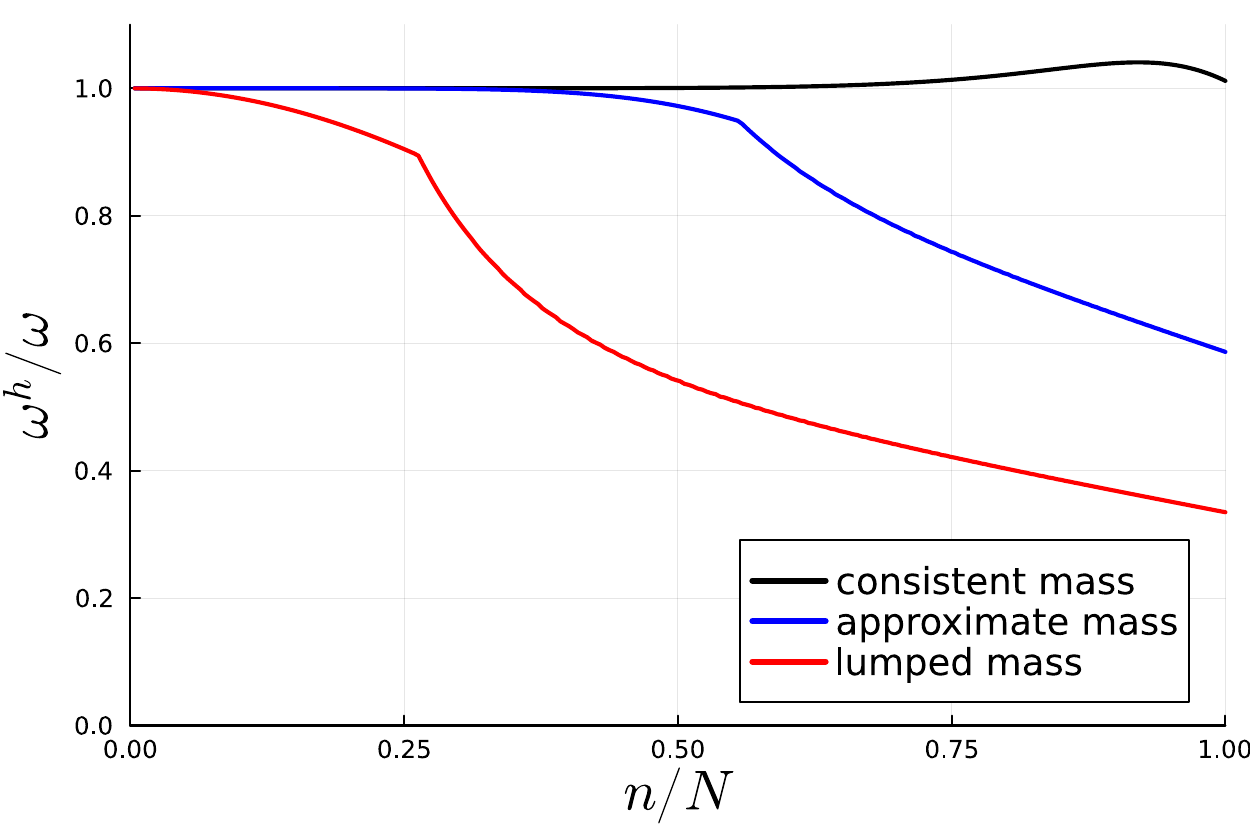}} 	\\
	\subfloat[$p=4$]{\includegraphics[trim = 0cm 0cm 0cm 0cm, clip,width=0.48\textwidth]{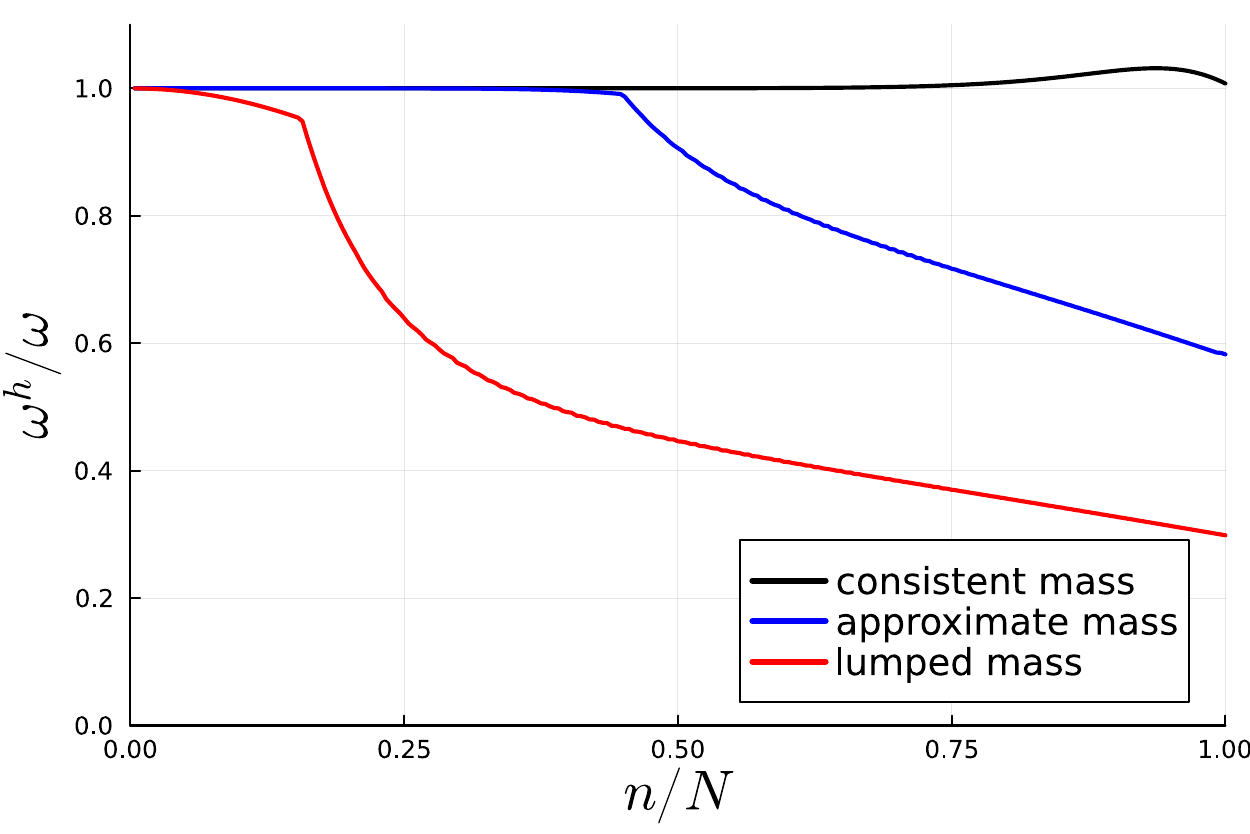}}
	\subfloat[$p=5$]{\includegraphics[trim = 0cm 0cm 0cm 0cm, clip,width=0.48\textwidth]{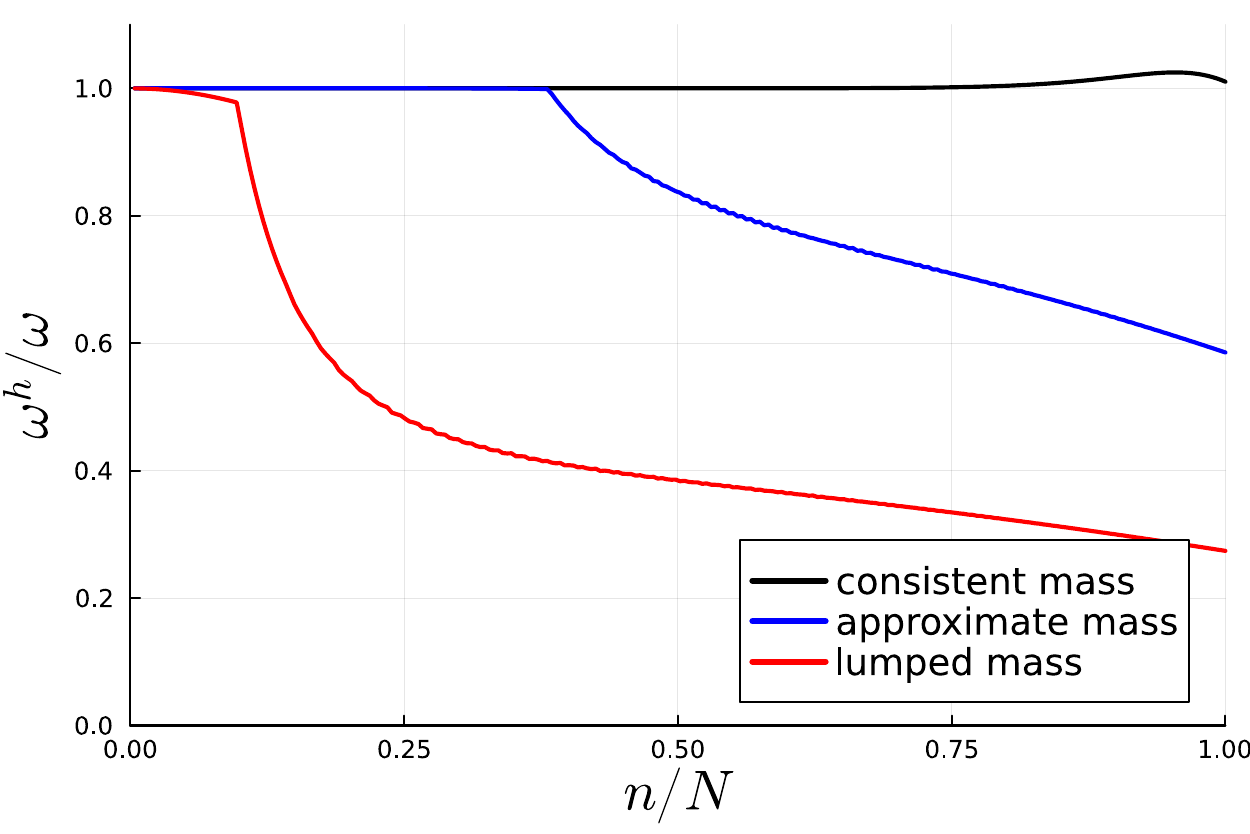}}	\\
	\caption{Normalized frequency spectra associated with consistent mass (black), higher order approximated mass (blue), and rowsum lumped mass (red). The results are obtained for polynomial degrees $p=2-5$ with $N=250$. The outlier frequencies have been removed using the technique \Bl presented \B in \cite{Hiemstra_outlier_2021}.}
\label{fig:spectra1}
\end{figure}

The gain in relative accuracy compared to \Bl the \B rowsum lumped mass becomes particularly apparent by investigating the relative errors accrued in the frequencies in a semilog plot, see Figure \ref{fig:spectra2}. The important frequencies obtained at low mode numbers using the higher order approximate mass are very close to those obtained using the consistent mass matrix. This is particularly the case for $p=2$ and $p=3$. For higher polynomial degrees ($p=4$ and $p=5$), we observe that the frequencies are much better approximated using the consistent mass matrix. Nevertheless, the gain in accuracy compared to lumped mass ranges from three to six orders in magnitude and the effect increases with polynomial degree. We note that high precisison arithmetic has been used in order to maintain and display results up to double precision arithmetic precision.
\begin{figure}[h]
	\centering
	\subfloat[$p=2$]{\includegraphics[trim = 0cm 0cm 0cm 0cm, clip,width=0.48\textwidth]{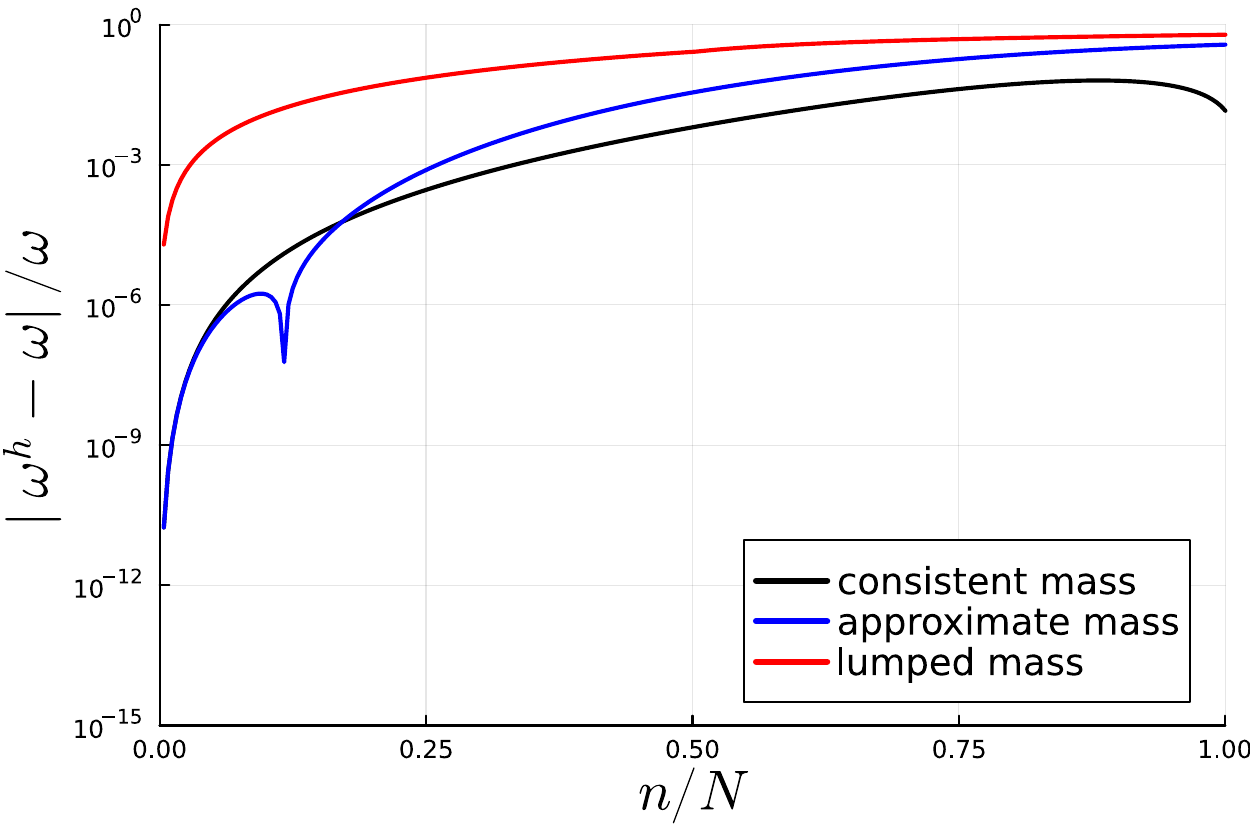}}
	\subfloat[$p=3$]{\includegraphics[trim = 0cm 0cm 0cm 0cm, clip,width=0.48\textwidth]{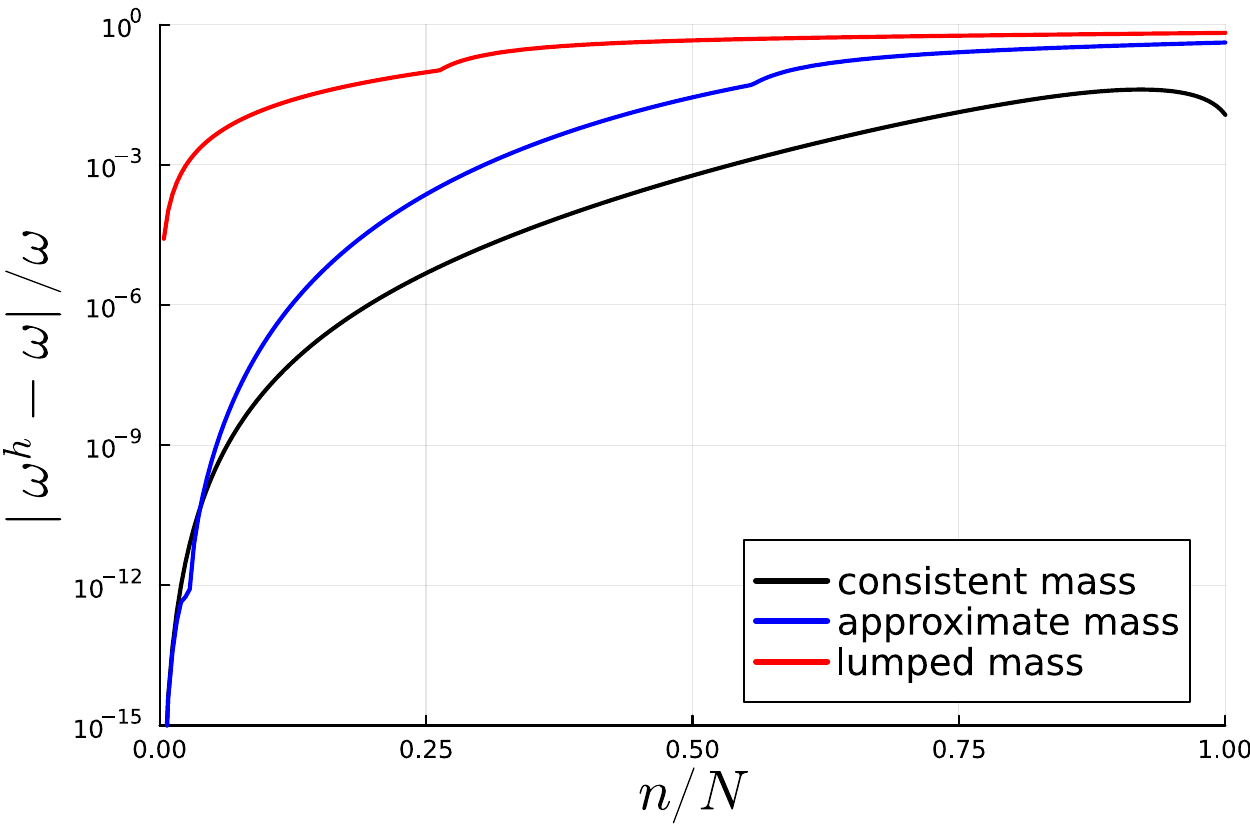}} 	\\
	\subfloat[$p=4$]{\includegraphics[trim = 0cm 0cm 0cm 0cm, clip,width=0.48\textwidth]{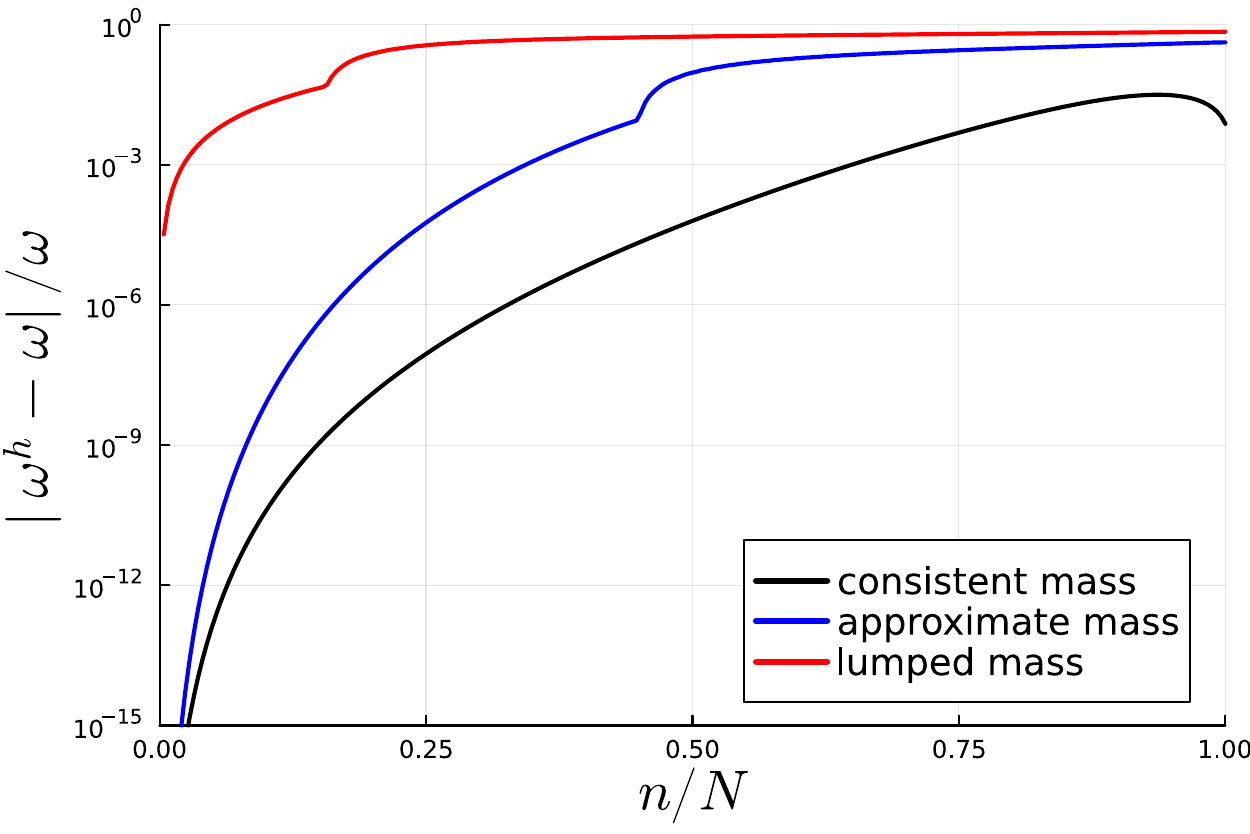}}
	\subfloat[$p=5$]{\includegraphics[trim = 0cm 0cm 0cm 0cm, clip,width=0.48\textwidth]{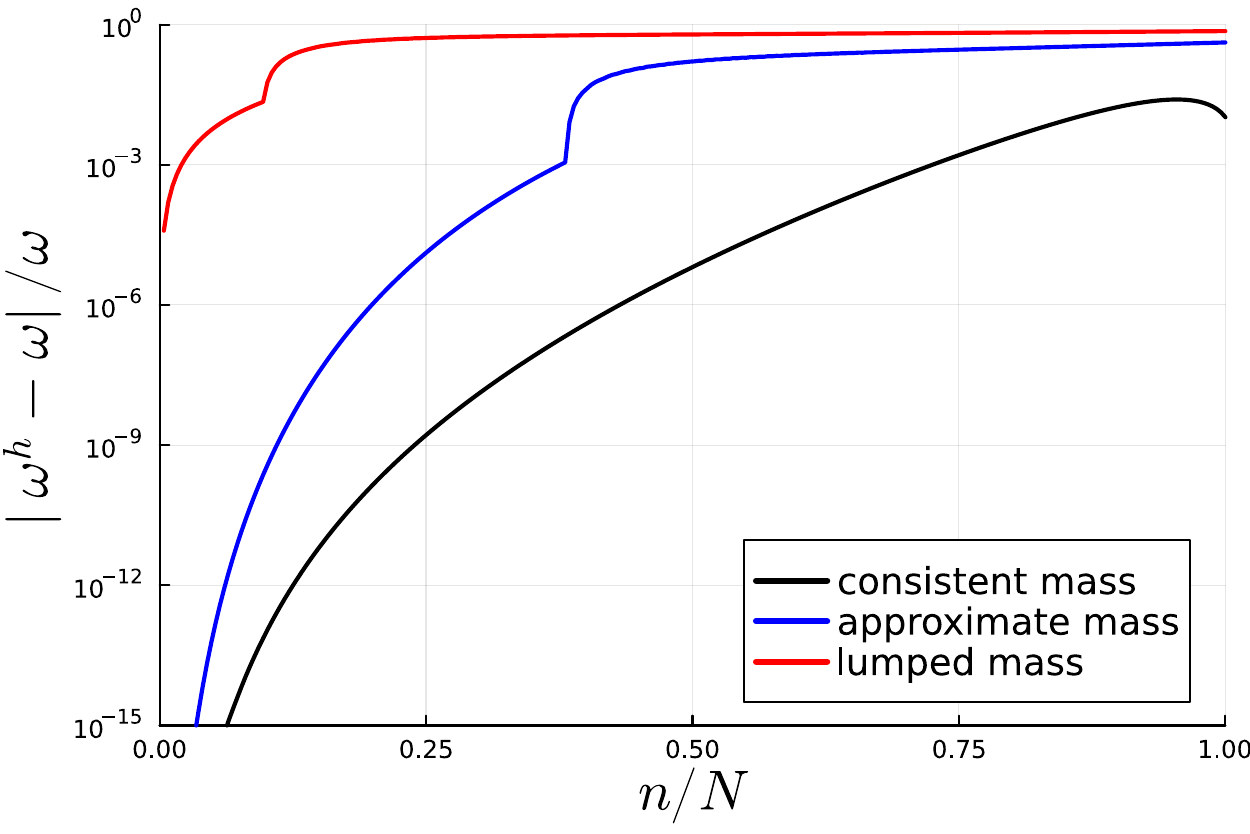}}	\\
	\caption{Normalized frequency errors associated with consistent mass (black), higher order approximated mass (blue), and rowsum lumped mass (red). The results are obtained for polynomial degrees $p=2-5$ with $N=250$. The outlier frequencies have been removed using the technique \Bl presented \B in \cite{Hiemstra_outlier_2021}.}
\label{fig:spectra2}
\end{figure}

\subsection{Linear analysis of free vibration of an annular membrane}
Consider the free vibration of an annular membrane with inner radius $a$ and outer radius $b$. The model is fixed at the boundaries. Let $J_4(r)$ denote the 4th Bessel function of the first kind and let $\lambda_k, \; k=1,2,\ldots$ denote its positive zeros. The radii of the annulus are chosen conveniently as certain zeros of $J_4(r)$. In particular $a = \lambda_2 \approx 11.065$ and $b = \lambda_4 \approx 17.616$. We define the following manufactured solution, which depends on both the radial coordinate $r$, the angular coordinate $\theta$, and time $t$
\begin{align}
	u(r,\theta,t) = J_4(r) \cdot \cos(\lambda_{2}  t) \cdot \cos(4\theta)
\end{align}
It may be verified that the function satisfies the differential equation for free vibration on the annulus with fixed boundary conditions at the inner and outer radii. Figure \ref{fig:problem description} shows the problem setup, the boundary conditions that are satisfied by $u$, and its initial condition $u(r,\theta,0)$.

We perform explicit dynamics with the consistent mass matrix, the higher order lumping technique, and standard rowsum lumping, and simulate one full period of the periodic function $u(r,\theta,t)$. In other words, the final time is $T = 2\pi / \lambda_{2}$. We use polynomial degrees $p = 3,4,5$ in combination with uniform refinement resulting in meshes with $n_{\text{elem}} = \List{ 8,16,32}$ elements in the radial coordinate and $2n_{\text{elem}} = \List{ 16,32,64}$ elements in the angular coordinate. We perform the simulations with and without the outlier removal technique presented in \cite{Hiemstra_outlier_2021}.
\begin{figure}
	\centering
	\begin{tikzpicture}
	\end{tikzpicture}
	\subfloat[Coarsest \Bezier mesh]{\includegraphics[trim = 0cm 0cm 0cm 0cm, clip,width=0.45\textwidth]{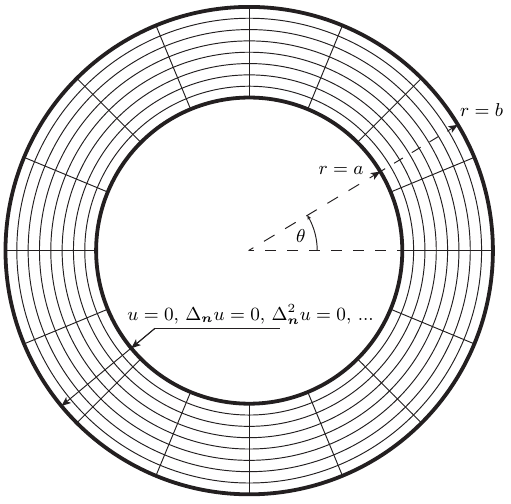}} \hspace{0.5cm}
	\subfloat[Initial displament field $u(r,\theta,0)$.]{\includegraphics[trim = 0cm 0cm 0cm 0cm, clip,width=0.45\textwidth]{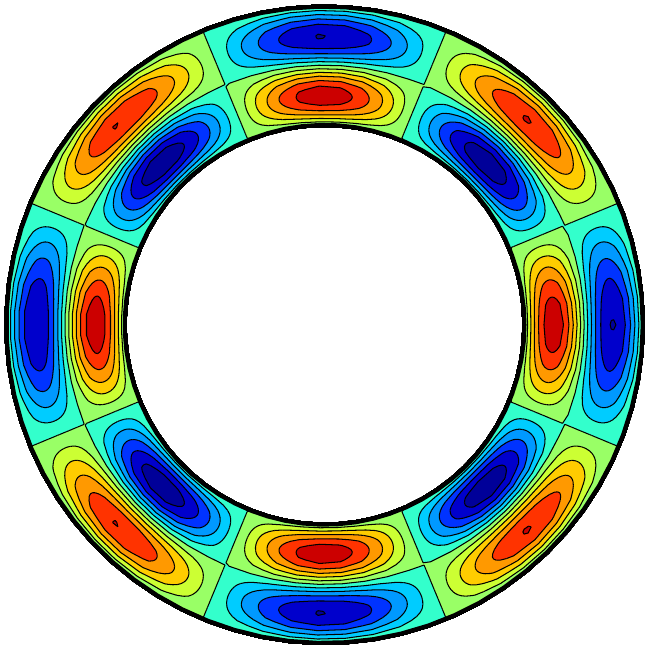}}
	\caption{Problem description for explicit dynamics on an annulus.}
\label{fig:problem description}
\end{figure}
For the higher order spatially accurate methods, namely with consistent mass or the higher order lumped mass, we discretize the semi-discrete equations in time using the \emph{RK4}-scheme when the polynomial degree is 3 or 4, and the \emph{RK6}-scheme when the polynomial degree is 5 (see \cite{Evans_corrector_scheme_2018} for an exposition  of the algorithms). With standard rowsum lumped mass we utilize the RK2 scheme. The time-step used is $50\%$ of the computed critical time-step of the respective time-discretization scheme ($\Delta_{\text{crit}}=C_{\text{max}} / \omega^h_{\text{max}}$, with $C_{\text{max}}=2.0$ for \emph{RK2}, $C_{\text{max}}=2.785$ for \emph{RK4} and $C_{\text{max}}=3.387$ for \emph{RK6}). The convergence behavior of the three methods, without and with outlier removal, is presented in Figure \ref{fig:convergence_annular_membrane} as a function of the square root of the total number of degrees of freedom. The higher order lumping technique preserves the accuracy that is obtained using \Bl the \B consistent mass. In addition, an outlier-free basis preserves optimal accuracy of the standard approach, however, using fewer degrees of freedom and larger time-steps. The improvement in the critical timestep can be observed in Figure \ref{fig:timestep}. Figure \ref{fig:timestep}a displays the improvement due to the choice of mass matrix (consistent, lumped, or higher order approximate mass) and Figure \ref{fig:timestep}b \Bl demonstrates \B the additional benefit of outlier removal. The compound effect is obtained by multiplying these numbers, which easily leads to a doubling of the critical timestep without loss of accuracy.

\begin{figure}
	\centering
	\subfloat[Without outlier removal]{\includegraphics[trim = 0cm 0cm 0cm 0cm, clip,width=0.65\textwidth]{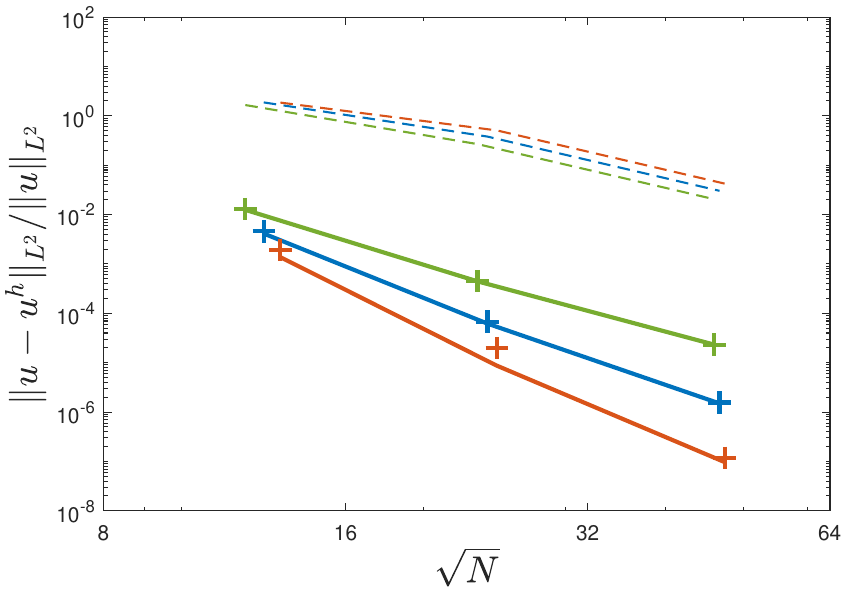}} \\
	\subfloat[With outlier removal]{\includegraphics[trim = 0cm 0cm 0cm 0cm, clip,width=0.65\textwidth]{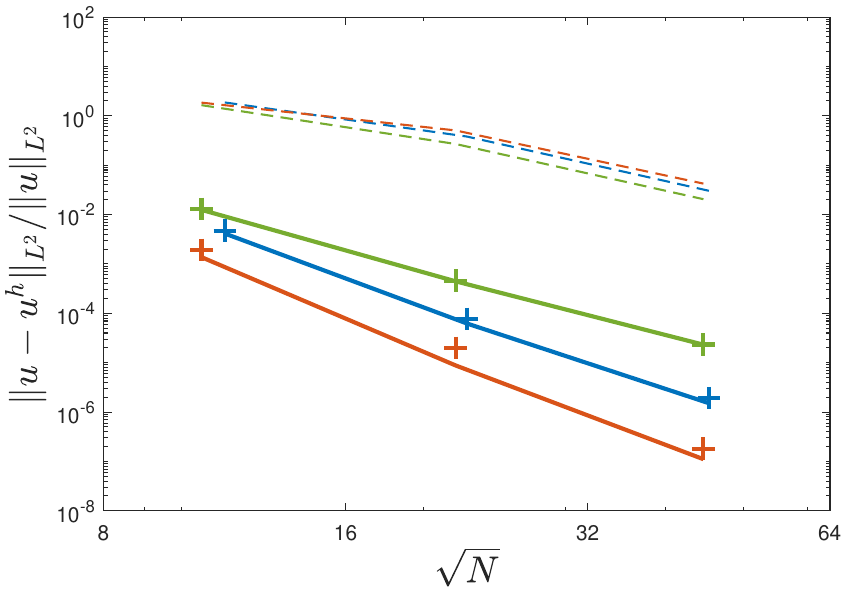}} \\
\begin{tikzpicture}
	\draw[green1,line width=2pt] (0,0) -- (0.5,0) node [right] {\scriptsize $\p=3$, consistent mass};
	\draw[blue1,line width=2pt] (5,0) -- (5.5,0) node [right] {\scriptsize $\p=4$, consistent mass};
	\draw[red1,line width=2pt] (10,0) -- (10.5,0) node [right] {\scriptsize $\p=5$, consistent mass};
	\draw[green1,line width=1pt] (0.25,-0.625) -- (0.25,-0.375);
	\draw[green1,line width=1pt] (0.125,-0.5) -- (0.375,-0.5);
	\node[green1] at (2.5,-0.5) {\scriptsize $\p=3$, higher order lumped mass};
	\draw[blue1,line width=1pt] (5.25,-0.625) -- (5.25,-0.375);
	\draw[blue1,line width=1pt] (5.125,-0.5) -- (5.375,-0.5);
	\node[blue1] at (7.5,-0.5) {\scriptsize $\p=4$, higher order lumped mass};
	\draw[red1,line width=1pt] (10.25,-0.625) -- (10.25,-0.375);
	\draw[red1,line width=1pt] (10.125,-0.5) -- (10.375,-0.5);
	\node[red1] at (12.5,-0.5) {\scriptsize $\p=5$, higher order lumped mass};
	\draw[dotted, green1,line width=1pt] (0,-1) -- (0.5,-1) node [right] {\scriptsize $\p=3$, lumped mass};
	\draw[dotted, blue1,line width=1pt] (5,-1) -- (5.5,-1) node [right] {\scriptsize $\p=4$, lumped mass};
	\draw[dotted, red1,line width=1pt] (10,-1) -- (10.5,-1) node [right] {\scriptsize $\p=5$, lumped mass};
	\node[green1] at (2,-1.5) {\scriptsize \emph{RK4}, $\Delta t = 0.5 \, \Delta t_{\text{crit}}$};
	\node[blue1] at (7,-1.5) {\scriptsize \emph{RK4}, $\Delta t = 0.5 \, \Delta t_{\text{crit}}$};
	\node[red1] at (12,-1.5) {\scriptsize \emph{RK6}, $\Delta t = 0.5 \, \Delta t_{\text{crit}}$};
\end{tikzpicture}
	\caption{Relative $L^2$ error in the vertical displacement field $u$ as a function of the square root of number of degrees of freedom $N$. Higher order mass lumping maintains the accuracy of the consistent mass, see Figure a. Outlier removal does not negatively affect the accuracy of the methods, see Figure b.}
\label{fig:convergence_annular_membrane}
\end{figure}

\begin{figure}
	\centering
	\subfloat[Benefit with respect to consistent mass]{\includegraphics[trim = 0cm 0cm 0cm 0cm, clip,width=0.48\textwidth]{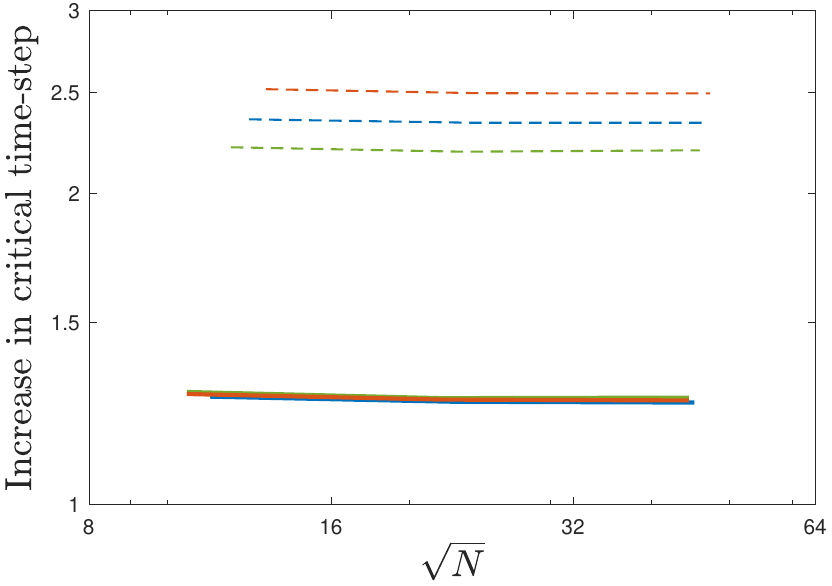}} \hspace{0.25 cm}
	\subfloat[Additional benefit of outlier removal]{\includegraphics[trim = 0cm 0cm 0cm 0cm, clip,width=0.48\textwidth]{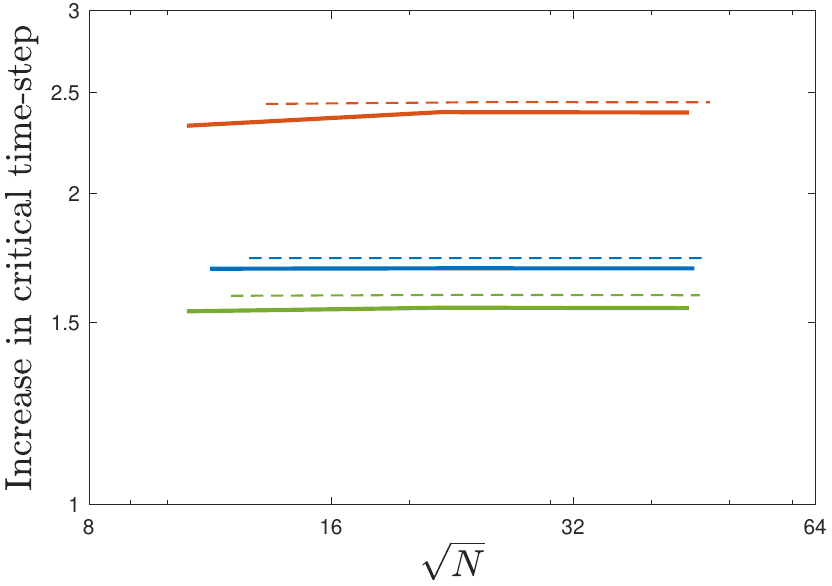}} \\
\begin{tikzpicture}	
	\draw[green1,line width=2pt] (0,0) -- (0.5,0) node [right] {\scriptsize $\p=3$, higher order lumped mass};
	\draw[blue1,line width=2pt] (5,0) -- (5.5,0) node [right] {\scriptsize $\p=4$, higher order lumped mass};
	\draw[red1,line width=2pt] (10,0) -- (10.5,0) node [right] {\scriptsize $\p=5$, higher order lumped mass};
	\draw[dotted,green1,line width=1pt] (0,-0.5) -- (0.5,-0.5) node [right] {\scriptsize $\p=3$, lumped mass};
	\draw[dotted,blue1,line width=1pt] (5,-0.5) -- (5.5,-0.5) node [right] {\scriptsize $\p=4$,  lumped mass};
	\draw[dotted,red1,line width=1pt] (10,-0.5) -- (10.5,-0.5) node [right] {\scriptsize $\p=5$,  lumped mass};
\end{tikzpicture}
	\caption{Increase of the critical time-step size of the (higher order) lumped mass compared with consistent mass (a) and the increase in timestep due to outlier removal (b). The compound effect of the chosen mass and outlier removal is obtained by multiplying these numbers.}
\label{fig:timestep}
\end{figure}
\section{Conclusion \label{sec:conclusion}}
This paper \Bl presents \B a comprehensive mathematical framework for explicit structural dynamics, utilizing approximate dual functionals and rowsum mass lumping. We \Bl demonstrate \B that the Petrov-Galerkin method with rowsum mass lumping proposed in \cite{nguyen2023} can also be interpreted as a Galerkin method with a customized mass matrix. Based on this observation, we \Bl prove \B mathematically that the customized mass technique does not compromise the asymptotic accuracy of the method. An essential improvement with respect to our prior work in \cite{nguyen2023} is the incorporation of Dirichlet boundary conditions while maintaining higher order accuracy. The mathematical results \Bl are \B substantiated by spectral analysis and validated through a two-dimensional linear explicit dynamics benchmark. Our method \Bl exhibits \B accuracy and robustness comparable with a Galerkin method employing consistent mass while retaining the explicit nature of lumped mass. The excellent coarse mesh accuracy of the approach paves the way for efficient explicit dynamics on coarse meshes with commensurable critical timesteps. In future work we intend to extend our results \Bl to \B non-linear problems and to multipatch configurations of complex geometric models found for example in the automotive industry.

\section*{Acknowledgments}

\section*{References}

\bibliographystyle{acm}
\bibliography{references}

\end{document}